\documentclass[12pt]{amsart}

\usepackage{amsfonts,amsmath,amsxtra,amsthm,tikz}
\usepackage{mathrsfs}
\usepackage{times}
\usepackage{anysize}

\marginsize{1in}{1in}{1in}{1.5in}

\usepackage{amscd}
\usepackage{amsthm}
\usepackage{comment}
\usepackage{amssymb}

\usetikzlibrary{arrows}

\newcommand{\Afund}{{\rm A}_0}  
\newcommand{\aBrn}{\widetilde{B}_n} 
\DeclareMathOperator{\Sh}{Sh} 
\newcommand{\Shikn}{Sh_n^k} 

\newcommand{\ZZ}{\mathbb{Z}_{\ge 0}}

\newtheorem{lemma}{Lemma}[section]

\newtheorem{theorem}{Theorem}[section]

\theoremstyle{definition}
\newtheorem{definition}{Definition}[section] 
\newtheorem{remark}{Remark}

\begin{document}
\author{Mikhail Mazin}
\address{Mathematics Department, Kansas State University.
Cardwell Hall, Manhattan, KS 66506}
\email{mmazin@math.ksu.edu}
\title{Multigraph Hyperplane Arrangements and Parking Functions.}
\date{}

\begin{abstract}
Back in the nineties Pak and Stanley introduced a labeling of the regions of a $k$-Shi arrangement by $k$-parking functions and proved its bijectivity. Duval, Klivans, and Martin considered a modification of this construction associated with a graph $G.$ They introduced the $G$-Shi arrangement and a labeling of its regions by $G$-parking functions. They conjectured that their labeling is surjective, i.e. that every $G$-parking function appears as a label of a region of the $G$-Shi arrangement. Later Hopkins and Perkinson proved this conjecture. In particular, this provided a new proof of the bijectivity of Pak-Stanley labeling in the $k=1$ case. We generalize Hopkins-Perkinson's construction to the case of arrangements associated with oriented multigraphs. In particular, our construction provides a simple straightforward proof of the bijectivity of the original Pak-Stanley labeling for arbitrary $k.$


\end{abstract}

\maketitle

\section*{Introduction}

Shi arrangement were introduced by Shi in \cite{Sh86} in his study of Kazhdan-Lusztig polynomials and cells of affine symmetric groups. In \cite{St98} Stanley studied the enumeration of the connected components of the complement to Shi arrangements (regions) with respect to distance from a base region, were the distance between two regions is defined as the number of hyperplanes of the arrangement separating them. The key step in Stanley's construction is a bijection between the set of regions and the set of parking functions, such that the sum of values of a parking function is equal to the distance from the corresponding region to the base region. This bijection is usually refered to as Pak-Stanley labeling. The construction of Pak and Stanley also applies to the extended Shi arrangements, also called $k$-Shi arrangements. The regions of $k$-Shi arrangements are labeled by $k$-parking functions.

The original proof of bijectivity of the Pak-Stanley labeling is very complicated. The objective of this paper is to generalize the labeling to a larger class of arrangements and to prove surjectivity of the generalized labeling (with an appropriately generalized notion of parking functions). In the case of $k$-Shi arrangements, the bijectivity follows then from comparing the number of regions and the number parking functions (both are computed in fairly elementary way). This provides a simple and straightforward proof of the bijectivity of the original Pak-Stanley labeling.

Our constructions were inspired by work of Hopkins and Perkinson (\cite{HP12}), were they prove surjectivity of Pak-Stanley labeling for arrangements associated with a graph. This was originally conjectured by Duval, Klivans, and Martin (\cite{DKM11}). Hopkins-Perkinson's results imply bijectivity of the original Pak-Stanley labeling for $k=1.$ We extend their construction in order to cover the case of general $k\ge 1.$ We also simplify the key step in their proof by replacing it with a transparent geometric argument.

The rest of the paper is organized as follows. We begin by reviewing the original Pak-Stanley labeling in Section \ref{Section: Pak-Stanley}. In Section \ref{Section: G-parking} we review the generalization associated with a graph. Finally, in Section \ref{Section: Multigraph} we present our construction and prove the surjectivity of the labeling.

\section*{Acknowledgements}

The author would like to thank Sam Hopkins for fruitful conversations at the FPSAC meeting in Chicago in Summer $2014,$ which motivated this work.

\section{Pak-Stanley Labeling}\label{Section: Pak-Stanley}
Let $V:=\{x_1+\ldots+x_n=0\}\subset\mathbb R^n.$ We will use the following notation: for any $i,j\in\{1,\ldots,n\}$ and $a\in\mathbb R,$ let 
$$
H_{ij}^a:=\{x_i-x_j=a\}\subset V.
$$ 

\begin{definition}
The {\it affine braid arrangement:} $\aBrn:=\left\{H_{ij}^l:0<j<i\le n,\ l\in\ZZ \right\}$ is the set of all such hyperplanes with integer right hand side. Connected components of the complement to $\aBrn$ are called {\it alcoves}. The {\it fundamental alcove:} is given by
$$
\Afund:=\{ 
\overline{x}\in V \mid
x_1>x_2>\ldots>x_n>x_1-1\}
$$
\end{definition}

Extended Shi arrangements are defined as subarrangements in $\aBrn:$

\begin{definition}
The {\it $k$-Shi arrangement} is given by 
$$
\Shikn:=\left\{H_{ij}^l:0<j<i\le n,\ -k<l\le k \right\}
$$ 
\end{definition}
See Figure \ref{Figure: Shi arrangements} for examples.

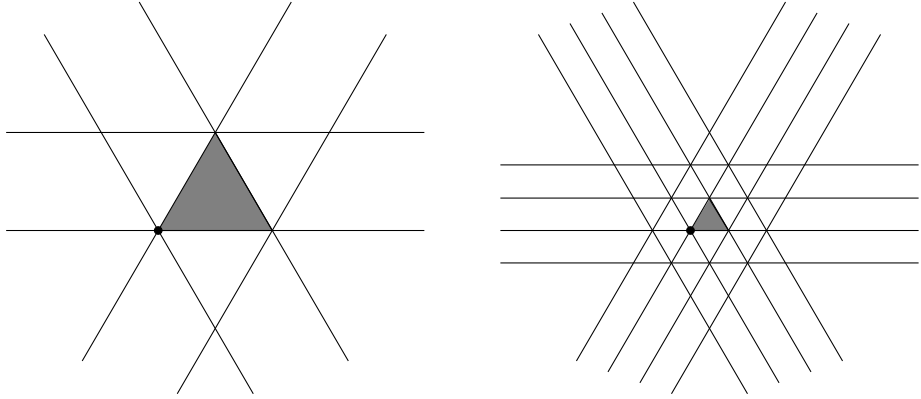
\begin{figure}
\begin{center}
\begin{tikzpicture}[scale=0.5]

\filldraw [fill=gray](4,0)--(7,0)--(5.5,2.6);

\filldraw [fill=black] (4,0) circle [radius=0.1];

\draw (0,0)--(11,0);
\draw (0,2.6)--(11,2.6);

\draw (2,-3.46)--(7.5,6.06);
\draw (4.5,-4.33)--(10,5.2);

\draw (1,5.2)--(6.5,-4.33);
\draw (3.5,6.06)--(9,-3.46);

\filldraw [fill=gray] (18,0)--(19,0)--(18.5,0.83);

\filldraw [fill=black] (18,0) circle [radius=0.1];

\draw (13,0)--(24,0);
\draw (13,0.86)--(24,0.86);

\draw (15.83,-3.75)--(21.33,5.77);
\draw (16.67,-4.04)--(22.17,5.49);

\draw (14.83,5.49)--(20.33,-4.04);
\draw (15.67,5.77)--(21.17,-3.75);

\draw (13,-0.86)--(24,-0.86);
\draw (13,1.74)--(24,1.74);

\draw (15,-3.46)--(20.5,6.06);
\draw (17.5,-4.33)--(23,5.2);

\draw (14,5.2)--(19.5,-4.33);
\draw (16.5,6.06)--(22,-3.46);

\end{tikzpicture}
\end{center}
\caption{Here $n=3.$ $1$-Shi arragnement on the left and $2$-Shi arrangement on the right.}\label{Figure: Shi arrangements}
\end{figure}

\begin{definition}
A function $f:\{1,\ldots,n\}\to\ZZ$ is called a {$k$-parking function} if for any subset $I\subset\{1,\ldots, n\}$ there exist a number $i\in I$ such that $f(i)\le k(n-\sharp I)$
\end{definition}

\begin{remark}
Alternatively, one can define $k$-parking functions as follows. Consider the Young diagram with row lengths given by the values $f(1),\ldots,f(n)$ put in non-increasing order. Then $f$ is a $k$-parking function if and only if the diagram fits under the diagonal of an $n\times kn$ rectangle. 
\end{remark}

Pak-Stanley labeling is a map from the set of regions of a $k$-Shi arrangement to the set of $k$-parking functions. It can be defined inductively using the following rules:
\begin{enumerate}
\item Label the fundamental alcove by the zero function.
\item When crossing a hyperplane $H_{ij}^l$ with $l>0$ in positive direction, increase the value $f(i)$ by $1.$
\item When crossing a hyperplane $H_{ij}^l$ with $l\le 0$ in positive direction, increase the value $f(j)$ by $1.$
\end{enumerate}

Here ``in positive direction'' means getting further away from the fundamental region. More directly, the parking function $\lambda_R$ corresponding to a region $R$ is given by the following formula:
$$
\lambda_R(a)=\sharp\left\{H_{ij}^l\in H_R:l>0, i=a\right\}+\sharp\left\{H_{ij}^l\in H_R:l\le 0, j=a\right\},
$$
where $H_R\subset\Shikn$ is the set of hyperplanes separating $R$ from $\Afund.$

See Figure \ref{Figure: PS labeling} for the Pak-Stanley labeling for $n=3$ and $k=1.$
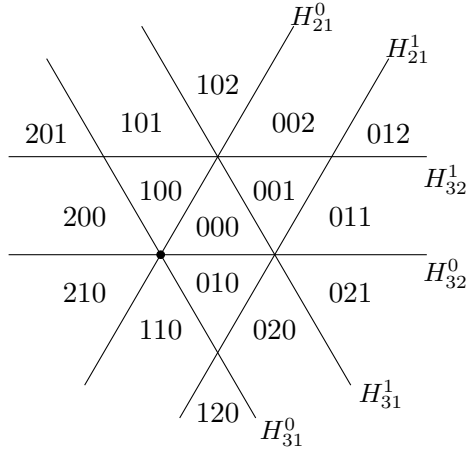
\begin{figure}
\begin{center}
\begin{tikzpicture}[scale=0.5]
\filldraw [fill=black] (4,0) circle [radius=0.1];

\draw (0,0)--(11,0);
\draw (0,2.6)--(11,2.6);

\draw (2,-3.46)--(7.5,6.06);
\draw (4.5,-4.33)--(10,5.2);

\draw (1,5.2)--(6.5,-4.33);
\draw (3.5,6.06)--(9,-3.46);

\draw (8,6.3) node {\footnotesize $H_{21}^0$};
\draw (10.5,5.45) node {\footnotesize $H_{21}^1$};

\draw (11.5,-0.5) node {\footnotesize $H_{32}^0$};
\draw (11.5,2) node {\footnotesize $H_{32}^1$};

\draw (7.2,-4.7) node {\footnotesize $H_{31}^0$};
\draw (9.8,-3.7) node {\footnotesize $H_{31}^1$};

\draw (5.5,0.75) node {\small $000$};
\draw (7,1.7) node {\small $001$};
\draw (4,1.7) node {\small $100$};
\draw (5.5,-0.75) node {\small $010$};
\draw (9,1) node {\small $011$};
\draw (2,1) node {\small $200$};
\draw (3.5,3.5) node {\small $101$};
\draw (7.5,3.5) node {\small $002$};
\draw (7,-2) node {\small $020$};
\draw (4,-2) node {\small $110$};
\draw (10,3.2) node {\small $012$};
\draw (5.5,4.5) node {\small $102$};
\draw (5.5,-4.2) node {\small $120$};
\draw (9,-1) node {\small $021$};
\draw (2,-1) node {\small $210$};
\draw (1,3.2) node {\small $201$};

\end{tikzpicture}
\end{center}
\caption{Pak-Stanley labeling for $n=3,$\ $k=1.$}\label{Figure: PS labeling}
\end{figure}

\begin{theorem}[\cite{St98}]
The Pak-Stanley labelling is a bijection between the set of regions of the $k$-Shi arrangement and the set of $k$-parking functions. 
\end{theorem}
The original proof is complicated, especially for $k>1.$

\section{$G$-parking Functions and $G$-Shi Arrangements}\label{Section: G-parking}

Let $G$ be a graph with the vertex set $\{1,\ldots,n\}.$ The $G$-Shi arrangement is the subarragement in $1$-Shi arrangement given by 
$$
\left\{H_{ij}^l\in\Sh_n^1:(i,j)\in E_G\right\}.
$$

The Pak-Stanley labeling is defined in the same way as before. See Figure \ref{Figure: GPS labeling} for an example. Note that the labeling is not injective: for example, on Figure \ref{Figure: GPS labeling} the labels in bold are the same. It is not hard to show that the labels satisfy the following condition:

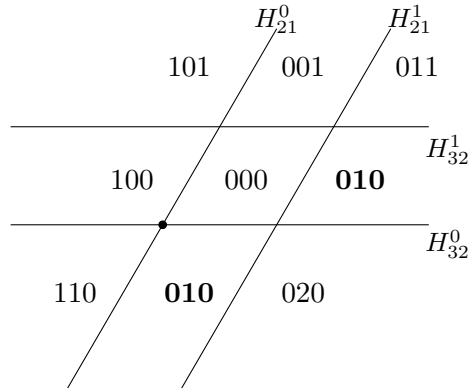
\begin{figure}
\begin{center}
\begin{tikzpicture}[scale=0.5]


\filldraw [fill=black] (4,0) circle [radius=0.1];

\draw (0,0)--(11,0);
\draw (0,2.6)--(11,2.6);

\draw (1.5,-4.33)--(7,5.2);
\draw (4.5,-4.33)--(10,5.2);


\draw (7,5.45) node {\footnotesize $H_{21}^0$};
\draw (10.5,5.45) node {\footnotesize $H_{21}^1$};

\draw (11.5,-0.5) node {\footnotesize $H_{32}^0$};
\draw (11.5,2) node {\footnotesize $H_{32}^1$};

\draw (6.2,1.2) node {\small $000$};
\draw (9.2,1.2) node {\small $\bf 010$};
\draw (3.2,1.2) node {\small $100$};

\draw (7.7,4.2) node {\small $001$};
\draw (10.7,4.2) node {\small $011$};
\draw (4.7,4.2) node {\small $101$};

\draw (4.7,-1.8) node {\small $\bf 010$};
\draw (7.7,-1.8) node {\small $020$};
\draw (1.7,-1.8) node {\small $110$};

\end{tikzpicture}
\end{center}
\caption{Pak-Stanley labeling for the graph $G:=\{1-2-3\}.$}\label{Figure: GPS labeling}
\end{figure}

\begin{definition}
A function $f:\{1,\ldots,n\}\to\ZZ$ is called a {$G$-parking function} if for any non-empty subset $I\subset\{1,\ldots,n\}$ there exists a vertex $i\in I$ such that the number of edges $(i,j)\in E_G$ such that $j\notin I$ is greater than or equal to $f(i).$
\end{definition}

\begin{theorem}[\cite{HP12}]
The labeling is a surjective map from the set of regions of a $G$-Shi arrangement to the set of $G$-parking functions.
\end{theorem}
This was first conjectured by Duval, Klivans, and Martin (\cite{DKM11}) and then proved by Hopkins and Perkinson. Note that this provides a new proof of bijectivity of the original Pak-Stanley labeling for $k=1.$

\begin{remark}
In fact, Hopkins and Perkinson proved more than that. It turns out that one can move the hyperplanes of the arrangement by changing the constant on the right hand side. The labeling will remain surjective as long as the fundamental region does not collapse.
\end{remark}

\section{Multigraphical Generalizations}\label{Section: Multigraph}

Let ${\mathcal A}$ be any finite arrangement of hyperplanes of the form $H_{ij}^a=\{x_i-x_j=a\}$ with $a>0.$
The region $R_0$ containing the origin is called {\it fundamental region.} Let $m_{ij}$ denote the number of hyperplanes of the form $H_{ij}^a$ in ${\mathcal A}.$ Note that we have $m_{ij}+m_{ji}$ parallel hyperplanes, $m_{ij}$ on one side of the origin and $m_{ji}$ on the other. The labeling defined in the same way as before. In fact, it is more symmetric in our notations. Let $R$ be a region, then $\lambda_R(i)$ equals to the number of hyperplanes $H_{ij}^l$ separating $R$ from $R_0.$ See Figure \ref{Figure: MGPS labeling} for an example.

\begin{figure}
\begin{center}
\begin{tikzpicture}[scale=0.5]
\filldraw [fill=black] (5,0.3) circle [radius=0.1];

\draw (0,0)--(11,0);

\draw (2,-3.46)--(7.5,6.06);

\draw (3.5,6.06)--(9,-3.46);

\draw (8,6.3) node {\footnotesize $H_{12}^a$};

\draw (11.5,-0.5) node {\footnotesize $H_{23}^b$};

\draw (9.8,-3.7) node {\footnotesize $H_{31}^c$};

\draw (5.5,0.75) node {\small $000$};
\draw (7,1.7) node {\small $001$};
\draw (4,1.7) node {\small $100$};
\draw (5.5,-0.75) node {\small $010$};
\draw (5.5,4.5) node {\small $101$};
\draw (9,-1) node {\small $011$};
\draw (2,-1) node {\small $110$};

\end{tikzpicture}
\end{center}
\caption{Here $a,b,c>0,$ and so $m_{12}=m_{23}=m_{31}=1,$\ $m_{32}=m_{21}=m_{13}=0,$ and the dot corresponds to the origin.}\label{Figure: MGPS labeling}
\end{figure}
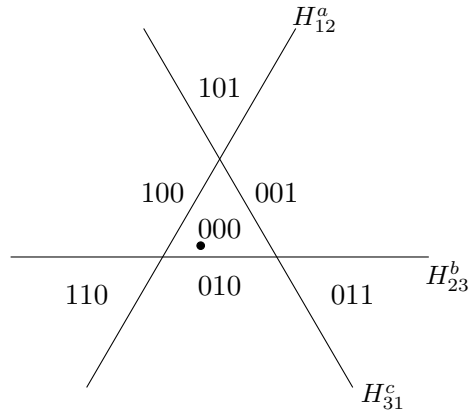

The oriented multigraph $G=G_{\mathcal A}$ is defined as follows. The vertex set is $V_G:=\{1,\ldots, n\},$ and the edge $i\rightarrow j$ has multiplicity $m_{ij}.$ Note that we allow arrows $(i\rightarrow j)$ and $(i\leftarrow j)$ with arbitrary non-negative multiplicities at the same time. In the example on Figure \ref{Figure: MGPS labeling}, one gets a circle graph (Figure \ref{Figure: cycle graph}).

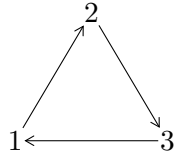
\begin{figure}
\begin{center}
\begin{tikzpicture}
\draw (0,0) node {\small $1$};
\draw (1,1.7) node {\small $2$};
\draw (2,0) node {\small $3$};

\draw [arrows={-angle 60}] (0.1,0.17)--(0.9,1.53);
\draw [arrows={-angle 60}] (1.1,1.53)--(1.9,0.17);
\draw [arrows={-angle 60}] (1.88,0)--(0.12,0);
\end{tikzpicture}
\end{center}
\caption{The graph corresponding to the arrangement in Figure \ref{Figure: MGPS labeling}.}\label{Figure: cycle graph}
\end{figure}


\begin{definition}
A function $f:\{1,\ldots,n\}\to\ZZ$ is called a {$G$-parking function} if for any non-empty subset $I\subset\{1,\ldots,n\}$ there exists a vertex $i\in I$ such that the number of edges $(i\rightarrow j)\in E_G,$ count with multiplicity, such that $j\notin I$ is greater than or equal to $f(i).$
\end{definition}

\begin{remark}
Note that at least one value of any $G$-parking function $f$ has to be zero. Indeed, take $I=\{1,\ldots,n\}.$ Then, according to the definition, there exists $i\in\{1,\ldots,n\},$ such that $f(i)=0.$ 
\end{remark}

\begin{lemma}
Let $R$ be a region of an arrangement ${\mathcal A}$ with the corresponding multigraph $G=G_{\mathcal A}.$ Then the label $\lambda_R$ is a $G$-parking function.
\end{lemma}

\begin{proof}
Consider the subarrangement ${\mathcal A}_R\subset {\mathcal A}$ consisting of hyperplanes of ${\mathcal A}$ that separate the region $R$ from the origin. Let $G_R=G_{{\mathcal A}_R}\subset G_{\mathcal A}$ be the corresponding sub-multigraph. The key observation is that $G_R$ cannot contain oriented cycles. Indeed, suppose that there is a cycle $i_1\rightarrow i_2\rightarrow\ldots\rightarrow i_k\rightarrow i_1.$ Then it follows that there exist hyperplanes $H_{i_1i_2}^{a_1},\ldots,H_{i_ki_1}^{a_k}\in {\mathcal A}$ separating region $R$ from the origin. Let ${\bf r}=(r_1,\ldots,r_n)\in R$ be a point. One gets
$$
r_{i_1}-r_{i_2}>a_1,
$$
$$
\vdots
$$
$$
r_{i_k}-r_{i_1}>a_k.
$$

Summing the inequalities up, one gets
$$
0>a_1+\ldots+a_k,
$$
which is impossible, because $a_i>0$ for all $i$ by construction. Contradiction.

Let us verify the $G$-parking function condition for $\lambda_R.$ Let $I\subset\{1,\ldots n\}$ be any non-empty subset. It follows that there exists $i\in I$ such that if $i\rightarrow j$ is an edge of $G_R,$ then $j\notin I.$ Indeed, otherwise $I$ contains an oriented cycle of $G_R.$ Note that $\lambda_R(i)$ is equal to the outgoing index of $i$ in $G_R,$ and all the outgoing edges from $i$ leave the subset $I.$ Therefore, $\lambda_R$ is a $G$-parking function.  
\end{proof}

\begin{theorem}\label{Theorem: surjectivity}
The labeling is surjective.
\end{theorem}

\begin{remark}
Note that arrangements with the same graph might have different combinatorics and even different number of regions. However, the labeling is always surjective. See Figure \ref{Figure: two arrangements same graph} for an example.
\end{remark}

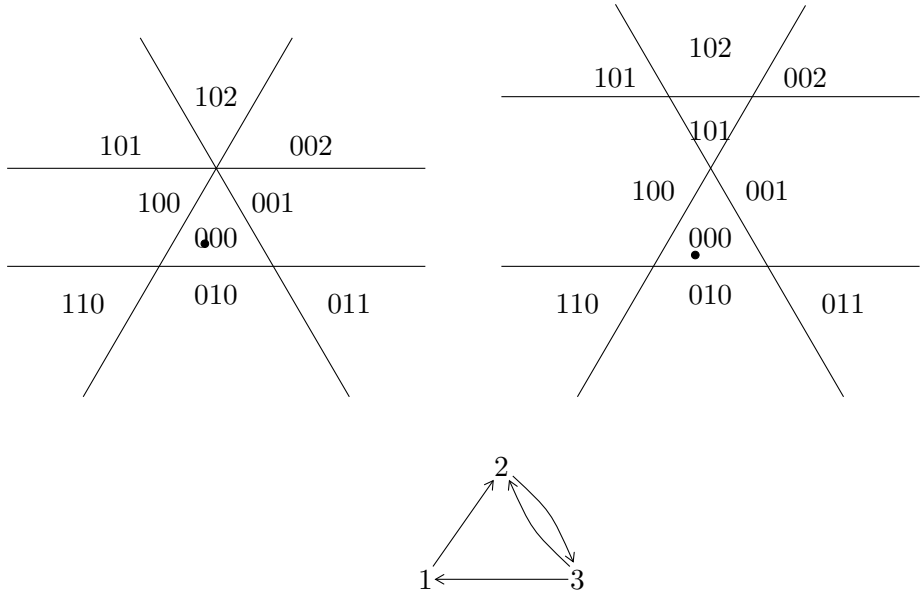
\begin{figure}  
\begin{center}
\begin{tikzpicture}[scale=0.5]


\filldraw [fill=black] (5.2,0.6) circle [radius=0.1];

\draw (0,0)--(11,0);
\draw (0,2.6)--(11,2.6);

\draw (2,-3.46)--(7.5,6.06);

\draw (3.5,6.06)--(9,-3.46);

\draw (5.5,0.75) node {\small $000$};
\draw (7,1.7) node {\small $001$};
\draw (4,1.7) node {\small $100$};
\draw (5.5,-0.75) node {\small $010$};
\draw (8,3.2) node {\small $002$};
\draw (5.5,4.5) node {\small $102$};
\draw (9,-1) node {\small $011$};
\draw (2,-1) node {\small $110$};
\draw (3,3.2) node {\small $101$};

\filldraw [fill=black] (18.1,0.3) circle [radius=0.1];

\draw (13,0)--(24,0);
\draw (13,4.5)--(24,4.5);

\draw (15,-3.46)--(21,6.92);

\draw (16,6.95)--(22,-3.46);

\draw (18.5,0.75) node {\small $000$};
\draw (20,2) node {\small $001$};
\draw (17,2) node {\small $100$};
\draw (18.5,-0.75) node {\small $010$};
\draw (21,5) node {\small $002$};
\draw (18.5,5.8) node {\small $102$};
\draw (18.5,3.6) node {\small $101$};
\draw (22,-1) node {\small $011$};
\draw (15,-1) node {\small $110$};
\draw (16,5) node {\small $101$};

\draw (11,-8.3) node {\small $1$};
\draw (13,-5.3) node {\small $2$};
\draw (15,-8.3) node {\small $3$};

\draw [arrows={-angle 60}] (13.3,-5.56) .. controls (14.3,-6.5) .. (14.9,-7.86);
\draw [arrows={-angle 60}] (14.8,-7.96) .. controls (13.8,-7) .. (13.2,-5.66);
\draw [arrows={-angle 60}] (11.2,-7.96)--(12.8,-5.66);
\draw [arrows={-angle 60}] (14.76,-8.3)--(11.24,-8.3);

\end{tikzpicture}
\end{center}
\caption{These two arrangements correspond to the same graph. Note that the labeling is surjective in both cases.}\label{Figure: two arrangements same graph}
\end{figure}

\begin{proof}[Proof of Theorem \ref{Theorem: surjectivity}]
Let $f:\{1,\ldots n\}\to \mathbb Z_{\ge 0}$ be a $G$-parking function. We want to find a region of ${\mathcal A}$ such that the corresponding label is equal to $f.$ We start from $R_0$ and cross exactly one hypersurface on each step. We proceed as follows. Suppose that we got to a region $R.$ Split the set of indexes in two subset 
$$
\{1,\ldots,n\}=I\sqcup J,
$$ 
where $I:=\{i:f(i)>\lambda_R(i)\}.$ On our next step we cross a hyperplane of the form $H_{ij}^a$ with $i\in I$ and $j\in J$ in positive direction, i.e. getting further away from the origin. As long as we can proceed, the following conditions will be satisfied:

\begin{enumerate}
\item For all $j\in J$ one has $f(j)=\lambda_R(j).$
\item The subset $I=\{i:f(i)>\lambda_R(i)\}$ is not increasing along the way. In particular, the subset $J$ is never empty (it was not empty in the beginning, because at least one value of $f$ is zero).
\item No hyperplane of the form $H_{pq}^a$ with $p,q\in I$ separates the region $R$ from $R_0$ at any step.
\end{enumerate}

Indeed, according to the algorithm, the values of the labeling can only increase along the way, and they cannot exceed $f(i).$ Condition $(3)$ then follows from condition $(2).$ All we need to prove now is that given these conditions, we can proceed until $I=\emptyset$ and $f=\lambda_R.$ Indeed, let us start from a point $(r_1,\ldots,r_n)\in R$ and move along the ray 

$$
x_i=r_i+\frac{t}{\sharp I},\ i\in I, 
$$
$$
x_j=r_j-\frac{t}{\sharp J},\ j\in J,
$$
$t\ge 0.$ If we hit a wall of $R,$ it has to be a hyperplane of the form $H_{ij}^a$ with $i\in I$ and $j\in J.$ If we never hit any walls, that means that all hyperplanes of that form are already between $R$ and $R_0.$ In which case the subset $I\subset\{1,\ldots,n\}$ violates the $G$-parking function condition. Indeed, since $f$ is a $G$-parking function, there should exist a vertex $i\in I$ such that 
$$
N_R(i):=\sharp\{(i\rightarrow j: j\in J)\}\ge f(i).
$$ 
However, there are no edges between vertices inside $I,$ so we get $\lambda_R(i)=N_R(i)\ge f(i),$ which contradicts the definition of the set $I.$  
\end{proof}

Note that this result provides a simple proof of the bijectivity of the original Pak-Stanley labeling for general $k.$ Indeed, given a $k$-Shi arrangement, one can change the coordinates by shifting the origin to the interior of the fundamental region. The resulting arrangement will fit into the framework of this chapter. The corresponding multigraph is the full graph with every edge having multiplicity $k$ in each direction. The $G$-parking function condition specializes to the $k$-parking function condition.

Finally, the number of regions of a $k$-Shi arrangement and the number of $k$-parking functions are both known to be equal to $(kn+1)^{n-1}.$ Therefore, in this special case, the labeling is not only surjective but, in fact, bijective.

\end{document}